\documentclass{amsart}

\usepackage{amsmath,amssymb,amsfonts,enumerate,amsthm, amscd}

\newcommand{\pd}{\mbox{pd}\,}
\newcommand{\id}{\mbox{id}\,}
\newcommand{\fd}{\mbox{fd}\,}

\newtheorem{theorem}{Theorem}[section]
\newtheorem{lemma}[theorem]{Lemma}

\newtheorem{corollary}[theorem]{Corollary}

\theoremstyle{definition}

\theoremstyle{remark}

\newtheorem{example}[theorem]{Example}

\theoremstyle{Definition and Notation}

\begin{document}
\bibliographystyle{amsplain}

\title[Homological dimensions of the amalgamated duplication of a ring...]{Homological dimensions of the amalgamated duplication of a ring along
a pure ideal}

\author{Mohamed Chhiti}
\address{Mohamed Chhiti\\Department of Mathematics, Faculty of Science and Technology of Fez, Box 2202, University S.M. Ben Abdellah Fez, Morocco.\\
chhiti.med@hotmail.com }
\author{Najib Mahdou}
\address{Najib Mahdou\\Department of Mathematics, Faculty of Science and Technology of Fez, Box 2202, University S.M. Ben Abdellah Fez, Morocco.\\
mahdou@hotmail.com }



\begin{abstract}

The aim of this paper is to study the classical  global and weak
dimensions of the  amalgamated duplication of a ring $R$ along a
pure ideal $I$.
\end{abstract}
\maketitle
\section{Introduction}
Throughout this paper all rings are commutative with identity
element and all modules are unitary.\\
Let $R$ be a ring, and let $M$ be an $R$-module. As usual we use
$\pd_R(M)$, $\id_R(M)$ and $\fd_R(M)$ to denote, respectively, the
classical projective dimension, injective dimension and flat
dimension of $M$. We use also $gldim(R)$ and $wdim(R)$ to denote,
respectively, the classical global and weak dimension of $R$.\\

For a nonnegative integer $n$, an $R$-module $M$ is $n$-presented
if there is an exact sequence $F_n\rightarrow F_{n-1} \rightarrow
...\rightarrow F_0 \rightarrow M\rightarrow 0$ in which each $F_i$
is a finitely generated free $R$-module. In particular,
``$0$-presented'' means finitely generated and ``$1$-presented''
means finitely presented. Set $\lambda_{R}(M)=\{ n/ M$ is
$n$-presented \}
 except if $M$ is not finitely generated. In this last case, we set $\lambda_{R}(M)= -1$. Not that
$\lambda_{R}(M)\geq n$ is a way to express the fact that $M$ is
$n$-presented. \\
Given nonnegative integers $n$ and $d$, a ring $R$ is called an
$(n,d)$-ring if every $n$-presented $R$-module has projective
dimension $\leq d$, and $R$ is called a weak $(n,d)$-ring  if
every $n$-presented cyclic $R$-module has projective dimension
$\leq d$. For instance, the $(0,1)$-domains are the Dedekind
domains, the $(1,1)$-domains are the Pr\"ufer domains and the
$(1,0)$-rings are the Von Neumann regular rings (see \cite{C, KM1,
KM2, M1, M2}). A commutative ring is called an $n$-Von Neumann
regular ring if it is an $(n,0)$-ring. Thus, the $1$-von Neumann
regular rings are the von Neumann regular rings (\cite[Theorem
1.3]{C}) . \\

The amalgamated duplication of a ring $R$ along an ideal $I$ is a
ring that is defined as the following subring with unit element
$(1,1)$ of $R\times R$:
\begin{eqnarray*}
   &R\bowtie I =\{(r,r+i)/r\in R,i\in I\}&
\end{eqnarray*}
This construction has been studied, in the general case, and from
the different point of view of pullbacks, by  D'Anna and Fontana
\cite{AF2}. Also, in \cite{AF1}, they  have considered the case of
the amalgamated duplication of a ring, in not necessarily
Noetherian setting, along a multiplicative canonical ideal in the
sense of \cite{HHP}. In \cite{A},  D'Anna has studied some
properties of $R\bowtie I $, in order to construct reduced
Gorenstein rings associated to Cohen-Macaulay rings and has
applied this construction to curve singularities. On the other
hand,  Maimani and  Yassemi, in \cite{MY}, have studied the
diameter and girth of the zero-divisor of the ring $R\bowtie I $.
Recently in \cite{CM}, the authors study some homological
properties of the rings $R\bowtie I$. Some references are
\cite{A, AF1, AF2, MY}. \\

Let $M$ be an $R$-module, the idealization $R\propto M$ (also
called the trivial extension), introduced by Nagata in 1956 (cf
 \cite{N}) is defined as the $R$-module $R\oplus M$ with multiplication
defined by $(r,m)(s,n)=(rs,rn+sm)$ (see
\cite{G, H, KM1, KM2}). \\

When $I^{2}=0$, the new construction $R\bowtie I $ coincides with
the idealization $R\propto I$. One main difference of this
construction, with respect to idealization is that the ring
$R\bowtie I $ can be a reduced ring (and, in fact, it is always
reduced if $R$ is a domain). \\
The first purpose of this paper is to study the classical global
and weak dimension of the amalgamated duplication of a ring $R$
along pure ideal $R$. Namely, we prove that if $I$ is a pure ideal
of $R$, then $wdim (R\bowtie I)= wdim (R)$.  Also,  we prove that
if $R$ is a coherent ring  and $I$ is a finitely generated pure
ideal of $R$, then $R\bowtie I$ is an $(1,d)$-ring provided the
local ring $R_{M}$ is an $(1,d)$-ring  for every maximal ideal $M$
of $R$. Finally, we give several  examples of rings which are
not weak $(n,d)$-rings (and so not $(n,d)$-rings) for each positive integers $n$ and $d$. \\

\section{main results}
Let $R$ be a commutative ring with identity element 1 and let $I$
be an ideal of $R$. We define $R\bowtie I =\{(r,s)/r,s\in R,s-r\in
I\}$. It is easy to check that $R\bowtie I$ is a subring with unit
element $(1,1)$,  of $R\times R$ (with the usual componentwise
operations) and that $R\bowtie I =\{(r,r+i)/r\in R,i\in I\}$.

It is easy to see that, if $ \pi_{i}$ $(i=1,2)$ are the
projections of $R\times R$ on $R$, then $\pi_{i}(R\bowtie I)=R$
 and hence if $O_{i}=ker(\pi_{i}\backslash R\bowtie I)$, then
$R\bowtie I/O_{i}\cong R$. Moreover $O_{1}=\{(0,i),i\in I\}$,
 $O_{2}=\{(i,0), i\in I\}$ and $O_{1}\cap O_{2}=(0)$.\\

Our first main result in this paper is given by the following
Theorem:
\begin{theorem}\label{1}
Let $R$ be a ring and   $I$ be a pure ideal of $R$. Then, $wdim
(R\bowtie I)= wdim (R)$.
\end{theorem}

To prove this Theorem we need some results.

\begin{lemma} \cite[Proposition 7]{A} \label{2}
 Let $R$  be a  ring  and let  $I$ be an ideal
 of $R$. Let $P$ be a prime ideal of $R$ and set:
\begin{itemize}
    \item $P_{0}=\{(p,p+i)/p\in P,i\in I\cap P\}$,
    \item $P_{1}=\{(p,p+i)/p\in P,i\in I\}$, and
    \item $P_{2}=\{(p+i,p)/p\in P,i\in I\}$.
\end{itemize}

\begin{enumerate}
    \item If $I\subseteq P$, then $P_{0}=P_{1}=P_{2}$ and $(R\bowtie I)_{P_{0}}\cong R_{P}\bowtie
    I_{P}$.
    \item If $I\nsubseteq P$, then $P_{1}\neq P_{2}$, $P_{1}\cap
    P_{2}=P_{0}$ and $(R\bowtie I)_{P_{1}}\cong R_{P} \cong (R\bowtie
    I)_{P_{2}}$.
\end{enumerate}
\end{lemma}

\begin{lemma}\label{2} Let $I$ be a non-zero flat ideal of a ring
$R$. For any $R$-module $M$ we have:
\begin{enumerate}
    \item  $fd_{R}(M)= fd_{R\bowtie I}(M\otimes_{R}(R\bowtie I))$.
    \item  $pd_{R}(M) =pd_{R\bowtie I}(M\otimes_{R}(R\bowtie I)$.

\end{enumerate}
\end{lemma}
\begin{proof}
 Note that the $R$-module $R\bowtie I$ is faithfully flat
since $I$ is flat.\\
 Firstly suppose that $fd_R(M)\leq n$ (resp.,
$pd_R(M)\leq n$) and pick an $n$-step flat (resp., projective)
resolution of $M$ over $R$ as follows:
    $$(\ast)\quad 0\rightarrow F_n\rightarrow F_{n-1} \rightarrow
...\rightarrow F_0 \rightarrow M\rightarrow 0.$$
 Applying the functor  $-\otimes_RR\bowtie I$ to $(\ast)$, we obtain
the  exact sequence of $(R\bowtie I)$-modules:
$$0\rightarrow
F_n\otimes_{R}(R\bowtie I)\rightarrow F_{n-1}\otimes_{R}(R\bowtie
I) \rightarrow ...\rightarrow F_0\otimes_{R}(R\bowtie I)
\rightarrow M\otimes_{R}(R\bowtie I)\rightarrow 0$$ Thus
$fd_{R\bowtie I}(M\otimes_{R}(R\bowtie I))\leq n$ (resp.,
$pd_{R\bowtie I}(M\otimes_{R}(R\bowtie I))\leq n$).

Conversely, suppose that $fd_{R\bowtie I}(M\otimes_{R}(R\bowtie
I))\leq n$ (resp., $pd_{R\bowtie I}(M\otimes_{R}(R\bowtie I))\leq
n$. Inspecting \cite[page 118]{CE} and since
$Tor^{k}_{R}(M,R\bowtie I)=0$
 for each  $k\geq1$, we conclude that  for any
$R$-module $N$ and each  $k\geq 1$ we have:

\begin{description}
    \item[(1)] $Tor^{k}_{R}(M,N\otimes_{R} (R\bowtie I))\cong Tor^{k}_{R\bowtie I}(M\otimes_{R} (R\bowtie I),
N\otimes_{R} (R\bowtie I))$
    \item[(2)] $Ext^{k}_{R}(M,N\otimes_{R}
(R\bowtie I))\cong Ext^{k}_{R\bowtie I}(M\otimes_{R} (R\bowtie
I),N\otimes_{R} (R\bowtie I))$
\end{description}

On the other hand $Tor_R^k(M,N)$ and $Ext_R^k(M,N)$ are direct
summands of $Tor^{k}_{R}(M,N\otimes_{R} (R\bowtie I))$ and
$Ext^{k}_{R}(M,N\otimes_{R} (R\bowtie I))$ respectively. Then, we
conclude that $fd_R(M)\leq n$ (resp., $pd_R(M)\leq n$) and this
finish the proof of this result.
\end{proof}
One direct consequence of this Lemma is:

\begin{corollary}\label{coro}

Let $I$ be a non-zero flat ideal of a ring $R$. Then:
\begin{enumerate}
    \item  $wdim(R)\leq wdim(R\bowtie I)$.
    \item  $gldim(R)\leq gldim(R\ltimes I)$.
\end{enumerate}
\end{corollary}

\begin{proof}[Proof of Theorem \ref{1}]

The inequality $wdim(R)\leq wdim(R\bowtie I)$ holds directly from
Corollary \ref{coro} since $I$ is pure and then flat. So, only the
other inequality need a proof.\\
Using \cite[Theorem 1.3.14]{G}  we have: $$(\intercal)\quad wdim
(R\bowtie I)= sup \{wdim ((R\bowtie I)_{M})|M\;is \; a\; maximal\;
ideal\; of\; R\bowtie I \}.$$ Let $M$ be an arbitrary maximal
ideal of $R\bowtie I$ and set $m:=M\cap R$. Then necessarily $M\in
\{ M_{1},M_{2}\}$ where $M_{1}=\{(r,r+i)/r\in m,i\in I\}$ and
$M_{2}=\{(r+i,r)/r\in m,i\in I\}$ (by \cite[Theorem 3.5]{AF2}). On
the other hand, $I_m\in \{0,R_m\}$ since $I$ is pure and $m$ is
maximal in R (by \cite[Theorem 1.2.15]{G}). Then, testing all
cases of Lemma \ref{2}, we resume two cases;
\begin{enumerate}
    \item $(R\bowtie I)_{M}\cong R_{m}$ if $I_m=0$ or $I\nsubseteq
    m$.
    \item $(R\bowtie I)_{M} \cong R_{m}\times R_m$ if $I_m=R_m$ or
    $I\subseteq m$.
\end{enumerate}

Hence, we have $wdim((R\bowtie I)_M)=wdim(R_m)\leq wdim(R)$. So,
the desired inequality follows from the equality $(\intercal)$.
\end{proof}
\begin{corollary}\label{3} Let $I$ be a finitely generated pure ideal of a ring  $R$. Then $R$ is a semihereditary ring if, and only
if, $R\bowtie I$ is a semihereditary ring.
\end{corollary}
\begin{proof} Follows immediately from Theorem \ref{1} and \cite[Theorem
3.1]{CM}.
\end{proof}

Recall that a ring $R$ is called  Gaussian  if $c(fg)=c(f)c(g)$
for every polynomials $f,g\in R[X]$, where $c(f)$ is the content
of $f$, that is, the ideal of $R$ generated by the coefficients of
$f$. See for instance \cite{G_{1}}.

\begin{corollary} Let $R$ be a reduced ring  and let  $I$ be  a pure
ideal of $R$. Then $R$ is  Gaussian  if, and only if, $R\bowtie I$
is   Gaussian.
\end{corollary}
\begin{proof} Follows immediately from Theorem \ref{1} , \cite[Theorem
2.2]{G_{1}} and  \cite[Theorem 3.5(a)(vi)]{AF2}.
\end{proof}

By the fact that every ideal over a Von Neumann regular ring is
pure, we conclude from Theorem \ref{1} the following Corollary
which have already proved in \cite{CM} with different methods.

\begin{corollary}Let $R$ be a ring  and let  $I$ be an
ideal of $R$. If $R$ is a Von Neumann regular ring, then so is
$R\bowtie I$.
\end{corollary}
If the ring $R$ is Noetherian the global and weak dimensions
coincide. Hence, Theorem \ref{1} can be writing as follows:

\begin{corollary} If $I$ is a pure ideal of a Noetherian ring $R$, then
 $gdim (R\bowtie I)= gdim (R)$.
\end{corollary}

A simple example of Theorem  \ref{1} is given by introducing the
notion of the trace of modules. Recall that if $M$ is an
$R$-module, the trace of $M$, $Tr(M)$, is the sum of all images of
morphisms $M\rightarrow R_{R}$ (see \cite{MPR}). Clearly $Tr(M)$
is an ideal of $R$.
\begin{example} If $M$ is a projective module over a ring $R$, then $wdim (R\bowtie Tr(M))=wdim
(R)$.
\end{example}
\begin{proof}
Clear since $Tr(M)$ is a pure ideal whenever $M$ is projective (by
\cite[pp. 269-270]{V}).
\end{proof}

Now, we study the transfer of an $(1,d)$-property.

\begin{theorem}\label{20}Let $R$ be a coherent ring such that for every maximal ideal $m$ of $R$
the local ring $R_{m}$ is an $(1,d)$-ring, and let  $I$ be a
finitely generated pure ideal of $R$. Then $R\bowtie I$ is an
$(1,d)$-ring.
\end{theorem}
\begin{proof}
Using \cite[Theorem 3.2]{C} and \cite[Theorem 3.1]{CM}, we have to
prove that for any maximal ideal $M$ of $R\bowtie I$, the ring
$(R\bowtie I)_M$ is an $(1,d)$-ring. So, let $M$ be such ideal and
set $m:=M\cap R$. From the proof of Theorem \ref{1}, we have two
possible cases:

\begin{enumerate}
    \item $(R\bowtie I)_{M}\cong R_{m}$ if $I_m=0$ or $I\nsubseteq
    m$.
    \item $(R\bowtie I)_{M} \cong R_{m}\times R_m$ if $I_m=R_m$ or
    $I\subseteq m$.
\end{enumerate}
So, by the hypothesis conditions, $(R\bowtie I)_M$ is an
$(1,d)$-ring since $R_m$ is it, as desired.
\end{proof}
By the fact that every ideal over a semisimple ring is pure we
conclude from Theorem \ref{20} the following Corollary.
\begin{corollary}Let $R$ be a ring  and let  $I$ be an
ideal of $R$. If $R$ is a semisimple ring, then so is $R\bowtie
I$.
\end{corollary}
Now, we give a wide class of rings which are not weak
$(n,d)$-rings (and so not $(n,d)$-rings) for each positive
integers $n$ and $d$.
\begin{theorem}\label{17}Let $R$ be a  ring and let $I$ be a proper ideal of $R$ which satisfies the following condition:

\begin{enumerate}
    \item $R_{m}$ is a domain for
  every maximal ideal $m$ of $R$.
    \item $I_{m}$ is a principal proper ideal of
$R_{m}$ for every maximal ideal $m$ of $R$.
\end{enumerate}
Then, $wdim(R\bowtie I) (=gldim(R\bowtie I))=\infty$.
\end{theorem}
\begin{proof}
Let $m$ be a maximal ideal of $R$ such that $I\subseteq m
\varsubsetneq R$. By Lemma \ref{2}, $R_m\bowtie I_m= (R\bowtie
I)_M$ where $M=\{(p,p+i)|p\in m,i\in I\}$. From  \cite[Theorem
2.13]{CM} and by the hypothesis conditions, we have $wdim
(R\bowtie I)_{M}= wdim (R_{m}\bowtie I_{m})=\infty$. Then, the
desired result follows from \cite[Theorem 1.3.14]{G}.
\end{proof}
The following example shows that the condition "$I_{m}$ is a
principal proper ideal of $R_{m}$ for every maximal ideal $m$ in
$R$" is necessary in Theorem \ref{17}.

\begin{example}Let $R$ be a Von Neumann regular ring and let $I$ be a proper
ideal of $R$. Then $wdim(R\bowtie I)=0$ since $(R\bowtie I)$ is a
Von Neumann regular ring, and $I_{m}$ is not a proper ideal  of
$R_{m}$ since $R_{m}$ is a field.
\end{example}


\bibliographystyle{amsplain}

\begin{thebibliography}{10}


\bibitem{C} D. L. Costa; \textit{Parameterising families of
      non-Noetherian rings}, Comm. Algebra, 22(1994), 3997-4011

      \bibitem{CE} H. Cartan and S. Eilenberg; \textit{Homological Algebra},
      Princeton Univ. Press. Princeton (1956).


      \bibitem{CM} M. Chhiti and N. Mahdou; \textit{Some homological
      properties of an amalgamated duplication of a ring along an
      ideal}, Submitted for publication. Available from math
      .AC/0903.2240 V1 12 mar 2009

       \bibitem{A} M. D'Anna; \textit{A construction of Gorenstein rings}, J. Algebra {\bf 306} (2006), no. 2, 507-519.

       \bibitem{AF1} M. D'Anna and M. Fontana; \textit{The amalgamated duplication
       of a ring along a multiplicative-canonical ideal}, Ark. Mat. {\bf 45} (2007), no. 2, 241-252.

       \bibitem{AF2} M. D'Anna and M. Fontana; \textit{An amalgamated duplication
      of a ring along an ideal: the basic properties}. J. Algebra Appl. {\bf 6} (2007), no. 3, 443-459.


      \bibitem{G} S. Glaz; \textit{Commutative Coherent Rings},
     Springer-Verlag, Lecture Notes in Mathematics, 1371 (1989).

     \bibitem{G_{1}} S. Glaz; \textit{The weak Dimenssions of Gaussian
     rings}, Proc.Amer. Maths. Soc. 133 (2005),2507-2513.

      \bibitem{H} J. A. Huckaba; \textit{Commutative Coherent Rings with Zero
     Divizors}. Marcel Dekker, New York Basel, (1988).

      \bibitem{HHP}W. Heinzer, J. Huckaba and I. Papick; \textit{m-canonical
      ideals in integral domains}, Comm.Algebra 26(1998), 3021-3043.

      \bibitem{KM1} S. Kabbaj and N. Mahdou; \textit{Trivial Extensions Defined by coherent-like condition},
      Comm.Algebra 32 (10) (2004), 3937-3953

      \bibitem{KM2} S. Kabbaj and N. Mahdou; \textit{Trivial extensions of local rings and a conjecture of Costa},
         Lecture  Notes in Pure and Appl. Math., Vol.231, Marcel Dekker, New York, (2003), 301-312.

     \bibitem{M1} N. Mahdou; \textit{On Costa's conjecture}, Comm.
     Algebra, 29 (7) (2001), 2775-2785.

     \bibitem{M2} N. Mahdou; \textit{On 2-Von Neumann regular rings}, Comm. Algebra 33 (10) (2005), 3489-3496.

     \bibitem{MPR} WM. McGovern, G. Puninski, P.
     Rothmaler; \textit{When every projective module is a direct sum of
     finitely generated modules}. J. Algebra 315 (2007). 454-481.

     \bibitem{MY} H. R. Maimani and S. Yassemi; \textit{Zero-divisor graphs of amalgamated duplication
      of a ring along an ideal}. J. Pure Appl. Algebra {\bf 212} (1) (2008), 168-174.

      \bibitem{N} M. Nagata; \textit{Local Rings}. Interscience, New york, (1962).

     \bibitem{PR} G. Puninski and P. Rothmaler, \textit{When every
     finitely generated flat module is projective}. J. Algebra 277 (2004), 542-558.


     \bibitem{V}W. V. Vasconcelas; \textit{Finiteness in projective ideals}.
     J. Algebra 25(1973). 269-278.
\end{thebibliography}

\end{document}